\documentclass[12pt]{elsarticle}
\usepackage{amscd,amssymb,amsmath,amsthm}
\usepackage[all]{xy}
\usepackage{array}
\usepackage{etoolbox}
\usepackage{mathtools}
\usepackage{tikz}
\usetikzlibrary{positioning,calc} 
\usepackage[most]{tcolorbox}
\usepackage{tikz-cd}

\makeatletter
\patchcmd{\ps@pprintTitle}{\footnotesize\itshape
        \hfill\today}{\relax}{}{}

\newdir{ >}{!/8pt/\dir{}*\dir{>}}
\newtheorem{theorem}{Theorem}[section]
\newtheorem*{main theorem}{Main Theorem}

\newtheorem{lemma}[theorem]{Lemma}

\newtheorem*{corollary*}{Corollary}
\newtheorem*{theorem*}{Theorem}

\newtheorem*{con7*}{Conjecture 7*}

\theoremstyle{definition}
\newtheorem{definition}[theorem]{Definition}
\newtheorem*{example*}{Example}

\newcommand\dela[1]{}

 \newcommand{\im}[1]{\text{Im }{(#1)}}

   \newcommand{\settheoremtag}[1]{
  \let\oldthetheorem\thetheorem
  \renewcommand{\thetheorem}{#1}
  \g@addto@macro\endtheorem{
    \global\let\thetheorem\oldthetheorem}
  }

\journal{}

\begin{document}

\begin{frontmatter}

\title{On the decidability of the integrability of finite groups}

 \author[IISER TVM]{Sathasivam Kalithasan}
\ead{sathasivam19@iisertvm.ac.in}
\author[IISER TVM]{Viji Z. Thomas\corref{cor1}}
\address[IISER TVM]{School of Mathematics,  Indian Institute of Science Education and Research Thiruvananthapuram,\\695551
Kerala, India.}
\ead{vthomas@iisertvm.ac.in}
\cortext[cor1]{Corresponding author. \emph{Phone number}: +91 8921458330}

\begin{abstract} An integral of a group $G$ is a group $H$ whose commutator subgroup is isomorphic to $G$. In this paper, we prove that the integrability of a finite group is a decidable problem.

\end{abstract}


\end{frontmatter}

 \section{Introduction}

The study of integrals of groups was first considered by Neumann in \cite{BHN}. However, little progress was made until recently, when the authors of \cite{ACC2019} and \cite{ACC2024} undertook a systematic study of these groups and obtained several nice results. One of the results they prove is that if a finite group $G$ is integrable, then its integral is finite. In both these papers, they list several open problems. The one that caught our attention is Question 9.2 in \cite{ACC2024}, and Problem 10.2 of \cite{ACC2019}. This problem was also highlighted in recent talks on the subject by Prof. Cameron \cite{Cameron2024Inverse, Cameron2025Inverse}. In particular, he poses the following question: Is it decidable whether a finite group $G$ is integrable? 

The purpose of the present paper is to answer this question. Our main result may be stated as follows.

\begin{theorem*}
    Let $G$ be a finite group. If $G$ is integrable, then $G$ has an integral $Q$ with \[\vert Q\vert \leq \big(\vert \text{Aut}(G)\vert \vert Z(G)\vert^{2\mu(G)} \big)^{d(Z(G))+1}.\]
\end{theorem*}

Our main theorem shows that the problem whether a finite group is integral is decidable by the following observation made by the authors of \cite{ACC2024}: 
\begin{quote}
\itshape
``compute all groups of order divisible by $|G|$ up to the bound, and decide for each group whether its derived group is isomorphic to $G$.''
\end{quote} 
\begin{corollary*}
The problem of whether a finite group $G$ is integrable is decidable.
\end{corollary*}

The authors of \cite{ACC2024} prove the following theorem which attempts to bound the order of some integral of a finite group $G$.
\begin{theorem*}[\cite{ACC2024}, Theorem~2.1]
Suppose there is a function $F$ from finite groups to natural numbers such that, if $G$ is an integrable finite group, then $F(G)$ is a bound for the exponent of the centre of some integral $H$ of $G$. Then there is a function $F^*$ from finite groups to natural numbers such that, if $G$ is an integrable finite group, then $G$ has an integral of order at most $F^*(G)$.
\end{theorem*}

They give the following beautiful example, which hints at the obstruction to their approach.

\begin{example*}\label{ex:obstruction}
For every $n \geq 3$, the group $C_2$ has an integral
\[
G_n = \langle a, b \mid a^{2^{n-1}} = b^2 = 1,\; b^{-1}ab = a^{-1} \rangle
\]
of order $2^n$, with $Z(G_n)$ cyclic of order $2^{n-2}$.
\end{example*}

A natural approach to bounding the order of an integral is to attempt to select an integral whose centre has smaller exponent. However, as stated in \cite{ACC2024},
\begin{quote}
\itshape
``Every proper subgroup or factor group of the group $G_n$ is abelian, so it is not at all clear how we could ``reduce'' it to a group with smaller cyclic center, although clearly such groups do exist.``
\end{quote}
This example shows that such a reduction is not straightforward in general. By employing cohomological methods, we overcome this obstruction.

 \section{Preliminary results}

 In the next theorem, we collect some of the results from the proof of Theorem~2.1 of \cite{ACC2024}.

 \begin{theorem}[\cite{ACC2024}, Theorem~2.1]\label{Integral with properties}
    Let $G$ be an integrable finite group. Then $G$ has an integral $H$ such that 
    \begin{itemize}
        \item[(i)] $H$ is a finite group.
        \item[(ii)] $H$ can be generated by some $t$ elements such that $t \le 2\mu(G)$, where $\mu(G)$ is the maximal size of a minimal generating set of $G$.
        \item[(iii)] $\vert H/Z(H)\vert \leq \vert \text{Aut }(G)\vert\;\vert Z(G)\vert^{2\mu(G)}$
        \item [(iv)] $d(Z(H))\leq d(Z(G))$, where $d(A)$ is the minimal number of generators of the abelian group $A$.
    \end{itemize}
\end{theorem}

\begin{definition}
    Let $A$ be a finite abelian group, and let $m$ be a natural number. The Omega subgroup of $A$ is defined as 
    \[\Omega_{m}(A)=\{a\in A\mid a^{m}=1\}.\]
\end{definition}
\section{Bounds on the size of the integral of group}
We do not necessarily need the next two lemmas for the main theorem, but we include them because they maybe of independent interest in the study of integrals of groups.
\begin{lemma}\label{intersection of quotient}
    Let $A$, $B$ be subgroups of a group $G$ and $C$ be a normal subgroup of $G$. If $A\cap B= C$, then $A/C\cap B/C=1.$ Moreover, If $A$ is normal in $G$, then there is an injective map $i: B/C\to G/A.$
\end{lemma}

\begin{lemma}\label{Basic properties of K}
    Let $H$ be integral of a group $G$. Set $K=C_H(G)$. Then
    \begin{enumerate}
    \item[(i)] $[K,G]=1.$ 
        \item[(ii)] $K\cap G=Z(G)$.
        \item [(iii)] $[H/K,H/K]\cong G/Z$.
        \item [(iv)] $K/Z(G)\cap [H,H]/Z(G)=1$
        \item [(v)] There exists an injective map $i:K/Z(G)\to H/[H,H].$
        \item [(vi)] $Z(G)$ is a characteristic subgroup of $H$ and $Z(G)\leqslant Z(K)$.
        \item [(vii)] $K$ is a normal subgroup of $H$ and $K/Z(G)$ is in the center of $H/Z(G)$
        \item [(viii)] $[K,H]\leqslant Z(G)$
        \item [(ix)]$Z(H)\leqslant Z(K)$
       \item[(x)] $K$ is nilpotent of class at most two.
        \item [(xi)] If $n$ is the exponent of $G$, then $K^n\leqslant Z(H).$
        
    \end{enumerate}
\end{lemma}
\begin{proof}
    \textit{(i)} follows easily from definition of $K$. \textit{(ii)} and \textit{(iii)} follows from the proof of \cite[Theorem 3.2]{ACC2024}. The parts \textit{(iv)} and \textit{(v)} follows from Lemma \ref{intersection of quotient}. 
    \par\textit{(vi)} Since $G$ is a characteristic subgroup of $H$ and $Z(G)$ is a characteristic subhgoup of $G$, we have $Z(G)$ is a characteristic subgroup of $H$. Using \textit{(i)} and \textit{(ii)}, we have $Z(G)\leqslant Z(K).$
    \par \textit{(vii)} Since $G$ is a normal subgroup of $H$, $K=C_H(G)$ is a normal subgroup of $H$. Moreover, let $kZ(G)\in K/Z(G)$ and $hZ(G)\in H/Z(G).$ Since ${[k,h]\in K\cap G=Z(G),}$ we have \[[kZ(G),hZ(G)]= [k,h]Z(G)=1 Z(G)\]and hence the proof.
    \par\textit{(viii)} Clearly $[K,H]\leqslant [H,H]=G.$ Since $K$ is normal, we have that  $[K,H]\leqslant K.$  Thus, $[K,H]\leqslant K\cap G=Z(G).$
    \par \textit{(ix)} Since $[Z(H),G]=1$, we have $Z(H)\leqslant C_H(G)=K$. Its easy to see $Z(H)\leqslant Z(K).$
    \par \textit{(x)}  Using Lemma \ref{Basic properties of K}\textit{(i)}, we have that\begin{align*}
        [[K,K],K]\leqslant [[H,H],K]=[G,K]=1
    \end{align*} 
    \par \textit{(xi)} Since $K=C_H([H,H])$, we have that $[kk_1,h]=[k,h][k_1,h]$ for all $k,k_1\in K, h\in H.$ Thus, it is suffices to show that $[k^n,h]=1$ for all ${k\in K}, {h\in H}.$ But $G=[H,H]$ has exponent $n$, so ${[k^n,h]=[k,h]^n=1.}$
    
\end{proof}

Now we come to the main result of this paper.
\begin{theorem}
    Let $G$ be a finite group. If $G$ is integrable, then $G$ has an integral $Q$ with \[\vert Q\vert \leq \big(\vert \text{Aut}(G)\vert \vert Z(G)\vert^{2\mu(G)} \big)^{d(Z(G))+1}.\]
\end{theorem}
\begin{proof}
    Let $G$ be a finite group with integral $H.$ We can assume that $H$ is a finite group satisfying properties listed in Theorem \ref{Integral with properties}. Set ${K=C_H(G)}$, ${\alpha=\vert H/Z(H)\vert}$. Let $\delta$ be a normalized 2-cocycle associated with the central extension
    \[
    1\longrightarrow Z(H)\longrightarrow H\longrightarrow H/Z(H)\longrightarrow 1.
    \]
     The key idea of the proof is to replace the kernel $Z(H)$ by a suitable abelian group whose exponent is known and the extension thus formed is also an integral of $G$. The proof involves two steps. In the first step, we enlarge the kernel $Z(H)$ by a suitable abelian group $X$ containing $Z(H)$ such that $d(Z(H)) = d(X)$. The cocycle $\delta$ is replaced by $i \circ \delta$, where $i \colon Z(H) \to X$ is the inclusion map. We then modify the cocycle $i\circ\delta$ by a suitable coboundary ${b\in B^2(H/Z, X)}$, so that $\im{i\circ\delta+b}\in \Omega_{\alpha }(X).$ In the second step, we show that the extension of $ H/Z(H)$ by $\Omega_{\alpha }(X)$ associated to ${(i\circ\delta)+b}$ is an integral of $G$. The advantage of making the kernel as $\Omega_{\alpha }(X)$  is that it has exponent $\alpha$, and hence its order can be determined.\\

\noindent\textbf{Step 1.} We first form the twisted product $E_\delta := Z(H) \times H/Z(H)$ with group operation
\begin{align}\label{twisted product}
(z_1, h_1)(z_2, h_2) := \bigl(z_1 +z_2 +\delta(h_1, h_2),\ h_1 h_2\bigr),  
\end{align}for all $z_1,z_2\in Z(H), h_1,h_2\in H/Z(H)$. The group $E_{\delta}$ is an integral of $G$ and we have a short exact sequence
\[
1 \longrightarrow Z(H) \longrightarrow E_\delta \longrightarrow H/Z(H) \longrightarrow 1.
\]The group operation of $E_\delta$ is given explicitly in terms of the cocycle $\delta$, which will be used in later computations. Note that the action of $H/Z(H)$ on $Z(H)$ is trivial and that $\alpha\delta\in B^2(H/Z(H),Z(H)).$ Thus, there exists a map 
\[
\Phi: H/Z(H) \longrightarrow Z(H)
\] 
such that 
\[
\alpha\,\delta(h_1, h_2) = \Phi(h_1)+\Phi(h_2)-\Phi(h_1 h_2), \quad h_1, h_2 \in H/Z(H).
\] 
Let $X$ be an abelian group such that $\alpha X \cong Z(H)$ and $d(X)=d(Z(H))$. From now on, we tacitly identify $\alpha X$ with $Z(H)$. For each $h \in H/Z(H)$, fix an $\alpha$-th root $\beta_{\Phi(h)} \in X$ of $\Phi(h)$, i.e., $\alpha \beta_{\Phi(h)}=\Phi(h)$. Define
\[
\Psi: H/Z(H) \longrightarrow X, \quad h \mapsto -\beta_{\Phi(h)}.
\]Then we have a 2-coboundary
\begin{align}\label{translative 2coboundary}
    b(h_1, h_2) := \Psi(h_1)+\Psi(h_2) -\Psi(h_1 h_2) \in B^2(H/Z(H), X).
\end{align}The inclusion $i:Z(H)\to X$ obtained via the identification $\alpha X\cong Z(H)$ induces a map  \[{Z^2(H/Z(H),Z(H))\to Z^2(H/Z(H),X)}.\] Set $\delta'= i\circ \delta$ and note that $\delta'\in  Z^2(H/Z(H),X)$. Similar to \eqref{twisted product}, define the twisted product $E_{\delta'}$ corresponding to $\delta'$. We claim $E_{\delta'}$ is also an integral of $G.$ To see this, consider the commutative diagram

\[
\begin{tikzcd}
1 \arrow[r] & Z(H) \arrow[r] \arrow[d,"i"] 
  & E_{\delta} \arrow[r] \arrow[d,"i_1"] 
  & H/Z(H) \arrow[r] \arrow[d,"i_2"] 
  & 1 \\
1 \arrow[r] & X \arrow[r] 
  & E_{\delta'} \arrow[r] 
  & H/Z(H) \arrow[r] 
  & 1
\end{tikzcd}
\]
where $i_1(z,h)=(i(z),h)).$ Since $i,i_2$ are injective, $i_1$ is also injective. Identifying $E_{\delta}$ with
its image, we have $E_{\delta}\leqslant E_{\delta'}$ and $[E_{\delta}, E_{\delta}]\leqslant [E_{\delta'}, E_{\delta'}]$. Let $(x_1,h_1), (x_2,h_2)\in E_{\delta'}.$ It is easy to see that
\begin{align}\label{Integral_computation}
     \ [(x_1,h_1), (x_2,h_2)]
    &=[(1,h_1),(1,h_2)]\in [E_{\delta}, E_{\delta}].
\end{align}Therefore, $[E_{\delta'}, E_{\delta'}]\leqslant [E_{\delta}, E_{\delta}]$ and hence $E_{\delta'}$ is an integral of $G.$ Let $b$ be the cocycle defined in $\eqref{translative 2coboundary}$ and define the twisted product $E_{\delta'+b}$ similar to $\eqref{twisted product}$. Since $E_{\delta'}\cong E_{\delta'+b}$, we have that $E_{\delta'+b}$ is also an integral of $G.$ Now we claim that $\im{\delta'+b}\in \Omega_{\alpha}(X)$. To this end, 
\begin{align*}
    \alpha(\delta'+b) (h_1,h_2) &= \alpha\delta'(h_1,h_2)+\alpha b(h_1,h_2)\\
    &= \Phi(h_1)+\Phi(h_2)-\Phi(h_1 h_2) +\alpha (\Psi(h_1)+\Psi(h_2) -\Psi(h_1 h_2))\\
    &=\Phi(h_1)+\Phi(h_2)-\Phi(h_1 h_2) +\alpha (-\beta_{\Phi(h_1)}-\beta_{\Phi(h_2)} +\beta_{\Phi(h_1 h_2)})\\
    &= 1
\end{align*}Thus, $\im{\delta'+b}\in \Omega_{\alpha}(X)$, and hence $\delta'+b\in Z^2(H/Z(H),\Omega_{\alpha}(X)).$ \\

\noindent\textbf{Step 2.} Let $Q$ be the twisted product of $H/Z(H)$ by $\Omega_{\alpha}(X)$ associated to $\delta'+b$. Consider the commutative diagram

\[
\begin{tikzcd}
1 \arrow[r] & \Omega_{\alpha}(X) \arrow[r] \arrow[d,"j_1"] 
  & Q \arrow[r] \arrow[d,"j_2"] 
  & H/Z(H) \arrow[r] \arrow[d,"j_3"] 
  & 1 \\
1 \arrow[r] & X \arrow[r] 
  & E_{\delta'+b} \arrow[r] 
  & H/Z(H) \arrow[r] 
  & 1
\end{tikzcd}
\]where $j_1$ is the inclusion map, $j_3$ is the identity map, and $j_2(x,h)=(x,h) $ for all $x\in \Omega_{\alpha}(X)$, $h\in H/Z(H)$. Note that $Q\leqslant E_{\delta'+b}$. Using calculations similar to \eqref{Integral_computation}, it is easy to see that $[Q,Q]= [ E_{\delta'+b}, E_{\delta'+b}]= G. $ Since $H$ satisfies the properties in Theorem \ref{Integral with properties} and noting that $d(\Omega_{\alpha}(X))\leq d(X)=d(Z(H))$, we have
\begin{align*}
    \vert Q\vert&= \vert H/Z(H)\vert \; \vert \Omega_{\alpha}(X)\vert\\
    &\leq    \vert H/Z(H)\vert \; \alpha^{d(\Omega_{\alpha}(X))}\\
    &\leq    \vert H/Z(H)\vert \; \alpha^{d(Z(G))}\\
    &\leq    \vert H/Z(H)\vert \vert H/Z(H)\vert^{d(Z(G))}\\
    &\leq \big(\vert \text{Aut}(G)\vert \vert Z(G)\vert^{2\mu(G)} \big)^{d(Z(G))+1}\\
\end{align*}

\end{proof}

\section*{Acknowledgements} 
Sathasivam K acknowledges the Ministry of Education,  Government of India, for the doctoral fellowship under the Prime Minister's Research Fellows (PMRF) scheme PMRF-ID 0801996. V. Z. Thomas acknowledges research support from ANRF, Government of India grant CRG/2023/004417.

\bibliographystyle{amsplain}
\bibliography{Reference}
\end{document}